\documentclass{amsart}
\usepackage{amsfonts}
\usepackage{amssymb}
\usepackage{amsmath}
\usepackage{cite}

\newtheorem{theorem}{Theorem}[section]
\theoremstyle{plain}

\newtheorem{corollary}[theorem]{Corollary}

\newtheorem{lemma}{Lemma}[section]

\numberwithin{equation}{section}
\input{tcilatex}

\begin{document}
\title[Geometric Properties of Bessel function derivatives]{Geometric
Properties of Bessel function derivatives}
\author{Erhan Deniz}
\address{Department of Mathematics, Faculty of Science and Letters, Kafkas
University, Kars, Turkey}
\email{edeniz36@gmail.com (E. Deniz), topkaya.sercan@hotmail.com (S.
Topkaya), mcaglar25@gmail.com (M. \c{C}a\u{g}lar)}
\author{Sercan Topkaya}
\author{Murat \c{C}a\u{g}lar}

\begin{abstract}
In this paper our aim is to find the radii of starlikeness and convexity of
Bessel function derivatives for three different kind of normalization. The
key tools in the proof of our main results are the Mittag-Leffler expansion
for $n$th derivative of Bessel function and properties of real zeros of it.
In addition, by using the Euler-Rayleigh inequalities we obtain some tight
lower and upper bounds for the radii of starlikeness and convexity of order
zero for the normalized $n$th derivative of Bessel function. The main
results of the paper are natural extensions of some known results on
classical Bessel functions of the first kind.
\end{abstract}

\maketitle

\section{Introduction}

Denote by $\mathbb{D}_{r}=\left\{ {z\in \mathbb{C}:\;\left\vert z\right\vert
<r}\right\} \quad (r>0)$ the disk of radius $r$ and let $\mathbb{D}=\mathbb{D%
}_{1}.$ Let $\mathcal{A}$ be the class of analytic functions $f$ in the open
unit disk $\mathbb{D}$ which satisfy the usual normalization conditions $%
f(0)={f}^{\prime }(0)-1=0.$ Traditionally, the subclass of $\mathcal{A}$
consisting of univalent functions is denoted by $\mathcal{S}.$ We say that
the function $f$ $\in \mathcal{A}$ is starlike in the disk $\mathbb{D}_{r}$
if $f$ is univalent in $\mathbb{D}_{r}$, and $f(\mathbb{D}_{r})$ is a
starlike domain in $%
\mathbb{C}
$ with respect to the origin. Analytically, the function $f$ is starlike in $%
\mathbb{D}_{r}$ if and only if $\func{Re}\left( \frac{zf^{\prime }(z)}{f(z)}%
\right) >0,$ $z\in \mathbb{D}_{r}.$ For $\beta \in \lbrack 0,1)$ we say that
the function $f$ is starlike of order $\beta $ in $\mathbb{D}_{r}$ if and
only if $\func{Re}\left( \frac{zf^{\prime }(z)}{f(z)}\right) >\beta ,$ $z\in 
\mathbb{D}_{r}.$ We define by the real number%
\begin{equation*}
r_{\beta }^{\ast }(f)=\sup \left\{ r\in \left( 0,r_{f}\right) :\func{Re}%
\left( \frac{zf^{\prime }(z)}{f(z)}\right) >\beta \;\text{for all }z\in 
\mathbb{D}_{r}\right\}
\end{equation*}%
the radius of starlikeness of order $\beta $ of the function $f$. Note that $%
r^{\ast }(f)=r_{0}^{\ast }(f)$ is the largest radius such that the image
region $f(\mathbb{D}_{r_{\beta }^{\ast }(f)})$ is a starlike domain with
respect to the origin.

The function $f$ $\in \mathcal{A}$ is convex in the disk $\mathbb{D}_{r}$ if 
$f$ is univalent in $\mathbb{D}_{r}$, and $f(\mathbb{D}_{r})$ is a convex
domain in $%
\mathbb{C}
$. Analytically, the function $f$ is convex in $\mathbb{D}_{r}$ if and only
if $\func{Re}\left( 1+\frac{zf^{\prime \prime }(z)}{f^{\prime }(z)}\right)
>0,$ $z\in \mathbb{D}_{r}.$ For $\beta \in \lbrack 0,1)$ we say that the
function $f$ is convex of order $\beta $ in $\mathbb{D}_{r}$ if and only if $%
\func{Re}\left( 1+\frac{zf^{\prime \prime }(z)}{f^{\prime }(z)}\right)
>\beta ,$ $z\in \mathbb{D}_{r}.$ The radius of convexity of order $\beta $
of the function $f$ is defined by the real number%
\begin{equation*}
r_{\beta }^{c}(f)=\sup \left\{ r\in \left( 0,r_{f}\right) :\func{Re}\left( 1+%
\frac{zf^{\prime \prime }(z)}{f^{\prime }(z)}\right) >\beta \;\text{for all }%
z\in \mathbb{D}_{r}\right\} .
\end{equation*}%
Note that $r^{c}(f)=r_{0}^{c}(f)$ is the largest radius such that the image
region $f(\mathbb{D}_{r_{\beta }^{c}(f)})$ is a convex domain.

The Bessel function of the first kind of order $\nu $ is defined by \cite[p.
217]{Olver} 
\begin{equation*}
J_{\nu }(z)=\sum_{m=0}^{\infty }\frac{\left( -1\right) ^{m}}{m!\Gamma (m+\nu
+1)}\left( \frac{z}{2}\right) ^{2m+\nu },\text{ \ \ }z\in 
\mathbb{C}
.
\end{equation*}%
Now, we consider the $n$th derivative of Bessel function of the first kind,
by%
\begin{equation*}
J_{\nu }^{(n)}(z)=\sum_{m=0}^{\infty }\frac{\left( -1\right) ^{m}\Gamma
(2m+\nu +1)}{m!2^{n}\Gamma (2m-n+\nu +1)\Gamma (m+\nu +1)}\left( \frac{z}{2}%
\right) ^{2m-n+\nu },\text{ \ \ }z\in 
\mathbb{C}
.
\end{equation*}

Here, it is important mentioning that for $n=0$ the $J_{\nu }^{(n)}$ reduce
to classical Bessel function $J_{\nu }$. Since the function $J_{\nu }^{(n)}$
is not belongs to $\mathcal{A}$, first we form some natural normalizations.
In this paper we focus on the following normalized forms%
\begin{eqnarray}
f_{\nu ,n}(z) &=&\left[ 2^{\nu }\Gamma (\nu -n+1)J_{\nu }^{(n)}(z)\right] ^{%
\frac{1}{\nu -n}},  \label{F1} \\
g_{\nu ,n}(z) &=&2^{\nu }\Gamma (\nu -n+1)z^{1+n-\nu }J_{\nu }^{(n)}(z), 
\notag \\
h_{\nu ,n}(z) &=&2^{\nu }\Gamma (\nu -n+1)z^{1+\frac{n-\nu }{2}}J_{\nu
}^{(n)}(\sqrt{z})  \notag
\end{eqnarray}%
where $\nu >n-1.$

The first studies on geometric properties of Bessel functions of first kind
was conducted in 1960 by Brown, Kreyszig and Todd \cite{Br,Kr}. They
determined the radius of starlikeness of the functions $f_{\nu ,0}(z)$ and $%
g_{\nu ,0}(z)$ for the case $\nu >0$. Recently, in 2014, Baricz et al. \cite%
{Baricz} and Baricz and Sz\'{a}sz \cite{Ba1} obtained, respectively, the
radius of starlikeness of order $\beta $ and the radius of convexity of
order $\beta $ for the functions $f_{\nu ,0}(z),$ $g_{\nu ,0}(z)$ and $%
h_{\nu ,0}(z)$ in the case when $\nu >-1.$ On the other hand, we know that
if $\nu \in (-2,-1),$ then the Bessel function has exactly two purely
imaginary conjugate complex zeros, and all the other zeros are real \cite[%
p.483]{Wat}. In 2015 , Sz\'{a}sz \cite{Sza} investigated the radius of
starlikeness of order $\beta $ for the functions $g_{\nu }(z)$ and $h_{\nu
}(z)$ in the case when $\nu \in (-2,-1)$ by using some inequalities. In the
same year, Baricz and Sz\'{a}sz \cite{Ba3} obtained the radius of convexity
of order $\beta $ for the functions $g_{\nu }(z)$ and $h_{\nu }(z)$ in the
case when $\nu \in (-2,-1).$ Later, in 2016, Baricz et al. \cite{Ba21}
determined the radius of $\alpha -$convexity of the same three functions for 
$\nu >-1$. After a year, \c{C}a\u{g}lar et al. \cite{Ca} extended it for the
case when $\nu \in (-2,-1)$. In 2017, Deniz and Sz\'{a}sz \cite{De}
determined the radius of uniform convexity of $f_{\nu ,0}(z),$ $g_{\nu
,0}(z) $ and $h_{\nu ,0}(z)$ for $\nu >-1$. They also determined necessary
and sufficient conditions on the parameters of these three normalized
functions such that they are uniformly convex in the unit disk. Moreover, in 
\cite{Ak, Ak1} authors determined tight lower and upper bounds for the radii
of starlikeness and convexity of the functions $g_{\nu ,0}(z)$ and $h_{\nu
,0}(z).$ The key tools in their proofs were some new Mittag-Leffler
expansions for quotients of Bessel functions of the first kind, special
properties of the zeros of Bessel functions of the first kind and their
derivatives, Euler-Rayleigh inequalities and the fact that the smallest
positive zeros of some Dini functions are less than the first positive zero
of the Bessel function of first kind.

Another study on Bessel functions investigate the properties of derivatives
and the zeros of these derivatives. In the last three decades the zeros of
the $n$th derivative of Bessel functions of the first kind for $n\in
\{1,2,3\}$ have been also studied by researchers like Elbert, Ifantis,
Ismail, Kokologiannaki, Laforgia, Landau, Lorch, Mercer, Muldoon,
Petropoulou, Siafarikas and Szeg\"{o}; for more details see the papers \cite%
{IS,KO} and the references therein. Very recently in 2018, Baricz et al. 
\cite{Ba4} obtained some results for the zeros of the $n$th derivative of
Bessel functions of the first kind for all $n\in 
\mathbb{N}
$ by using the Laguerre-P\'{o}lya class of entire functions and the
so-called Laguerre inequalities.

Motivated by the above results in this paper, we deal with the radii of
starlikeness and convexity of order $\beta $ for the functions $f_{\nu
,n}(z),\;g_{\nu ,n}(z)$ and $h_{\nu ,n}(z)$ in the case when $\nu >n-1$ for $%
n\in 
\mathbb{N}
.$ Also we determined tight lower and upper bounds for the radii of
starlikeness and convexity of these functions.

\section{Preliminaries}

In order to prove the main results we need the following preliminary results.

\begin{lemma}
\label{l1}\cite{Ba4} The following assertions are valid:

\begin{enumerate}
\item[a.] If $\nu >n-1,$ then $z\mapsto J_{\nu }^{(n)}(z)$ has infinitely
many zeros, which are all real and simple, expect the origin.

\item[b.] If $\nu >n,$ then the positive zeros of the $n$th and $(n+1)$th
derivative of $J_{\nu }$ are interlacing.

\item[c.] If $\nu >n-1,$ then all zeros of $z\mapsto (n-\nu )J_{\nu
}^{(n)}(z)+zJ_{\nu }^{(n+1)}(z)$ are real and interlace with the zeros of $%
z\mapsto J_{\nu }^{(n)}(z).$
\end{enumerate}
\end{lemma}

We will also need the following result, see \cite{Bi,Po}:

\begin{lemma}
\label{l11} Consider the power series $f(x)=\sum_{n=0}^{\infty }a_{n}x^{n}$
and $g(x)=\sum_{n=0}^{\infty }b_{n}x^{n}$, where $a_{n}\in 
\mathbb{R}
$ and $b_{n}>0$ for all $n\geq 0$. Suppose that both series converge on $%
(-r,r)$, for some $r>0$. If the sequence $\{a_{n}\diagup b_{n}\}_{n\geq 0}$
is increasing (decreasing), then the function $x\rightarrow f(x)\diagup g(x)$
is increasing (decreasing) too on $(0,r)$. The result remains true for the
power series%
\begin{equation*}
f(x)=\sum_{n=0}^{\infty }a_{n}x^{2n}\text{ and }g(x)=\sum_{n=0}^{\infty
}b_{n}x^{2n}.
\end{equation*}
\end{lemma}

\subsection{Zeros of hyperbolic polynomials and the Laguerre--P\'{o}lya
class of entire functions}

In this subsection, we recall some necessary information about polynomials
and entire functions with real zeros. An algebraic polynomial is called
hyperbolic if all its zeros are real. We formulate the following specific
statement that we shall need, see \cite{Ba2} for more details.

\begin{lemma}
\label{Le1} Let $p(x)=1-a_{1}x+a_{2}x^{2}-a_{3}x^{3}+\cdots
+(-1)^{n}a_{n}x^{n}=(1-x/x_{1})\cdots (1-x/x_{n})$ be a hyperbolic
polynomial with positive zeros $0<x_{1}\leq x_{2}\leq \cdots \leq x_{n},$
and normalized by $p(0)=1.$ Then, for any constant $C,$ the polynomial $%
q(x)=Cp(x)-xp^{\prime }(x)$ is hyperbolic. Moreover, the smallest zero $\eta
_{1\text{ }}$belongs to the interval $(0,x_{1})$ if and only if $C<0.$
\end{lemma}

By definition, a real entire function $\psi $ belongs to the Laguerre--P\'{o}%
lya class $\mathcal{LP}$ if it can be represented in the form%
\begin{equation*}
\psi (x)=cx^{m}e^{-ax^{2}+\beta x}\dprod\limits_{k\geq 1}\left( 1+\frac{x}{%
x_{k}}\right) e^{-\frac{x}{x_{k}}},
\end{equation*}%
with $c,\beta ,x_{k}\in 
\mathbb{R}
,$ $a\geq 0,$ $m\in 
\mathbb{N}
\cup \{0\}$ and $\sum x_{k}^{-2}<\infty .$ Similarly, $\phi $ is said to be
of type $\mathcal{I}$ in the Laguerre-P\'{o}lya class, written $\varphi \in 
\mathcal{LPI}$, if $\phi (x)$ or $\phi (-x)$ can be represented as%
\begin{equation*}
\phi (x)=cx^{m}e^{\sigma x}\dprod\limits_{k\geq 1}\left( 1+\frac{x}{x_{k}}%
\right) ,
\end{equation*}%
with $c\in 
\mathbb{R}
,$ $\sigma \geq 0,$ $m\in 
\mathbb{N}
\cup \{0\},$ $x_{k}>0$ and $\sum x_{k}^{-1}<\infty .$ The class $\mathcal{LP}
$ is the complement of the space of hyperbolic polynomials in the topology
induced by the uniform convergence on the compact sets of the complex plane
while $\mathcal{LPI}$ is the complement of the hyperbolic polynomials whose
zeros possess a preassigned constant sign. Given an entire function $\varphi 
$ with the Maclaurin expansion%
\begin{equation*}
\varphi (x)=\sum_{k\geq 0}\mu _{k}\frac{x^{k}}{k!},
\end{equation*}%
its Jensen polynomials are defined by 
\begin{equation*}
P_{m}(\varphi ;x)=P_{m}(x)=\sum_{k=0}^{m}\left( 
\begin{array}{c}
m \\ 
k%
\end{array}%
\right) \mu _{k}x^{k}.
\end{equation*}%
The next result of Jensen \cite{Je} is a well-known characterization of
functions belonging to $\mathcal{LP}$.

\begin{lemma}
\label{Le2} The function $\varphi $ belongs to $\mathcal{LP}$ $(\mathcal{LPI}
$, respectively$)$ if and only if all the polynomials $P_{m}(\varphi ;x)$, $%
m=1,2,...,$ are hyperbolic (hyperbolic with zeros of equal sign). Moreover,
the sequence $P_{m}(\varphi ;z\diagup n)$ converges locally uniformly to $%
\varphi (z)$.
\end{lemma}

The following result is a key tool in the proof of main results.

\begin{lemma}
\label{Le3} Let $\nu >n-1$ and $a<0.$ Then the functions $z\longmapsto
(2a-n+\nu )J_{\nu }^{(n)}(z)-zJ_{\nu }^{(n+1)}(z)$ can be represented in the
form%
\begin{equation*}
2^{n-1}\Gamma (\nu +1-n)\left( (2a-n+\nu )J_{\nu }^{(n)}(z)-zJ_{\nu
}^{(n+1)}(z)\right) =\left( \frac{z}{2}\right) ^{\nu -n}W_{\nu ,n}(z)
\end{equation*}%
where $W_{\nu ,n}$ is entire functions belonging to the Laguerre--P\'{o}lya
class $\mathcal{LP}$. Moreover, the smallest positive zero of $W_{\nu ,n}$
does not exceed the first positive zero $j_{\nu ,1}^{(n)}$ where $j_{\nu
,m}^{(n)}$ is the $m$th positive zero of $J_{\nu }^{(n)}(z)$ $(m\in 
\mathbb{N}
,$ $n\in 
\mathbb{N}
_{0})$.
\end{lemma}

\begin{proof}
It is clear from the infinite product representation of $z\longmapsto 
\mathcal{J}_{\nu }^{(n)}(z)=2^{\nu }\Gamma (\nu +1-n)\left( z\right) ^{n-\nu
}J_{\nu }^{(n)}(z)$\ that this function belongs to $\mathcal{LP}$. This
implies that the function $z\longmapsto \mathbb{J}_{\nu }^{(n)}(z)=\mathcal{J%
}_{\nu }^{(n)}(2\sqrt{z})$ belongs to $\mathcal{LPI}$. Then it follows form
Lemma \ref{Le2} that its Jensen polynomials 
\begin{equation*}
P_{m}(\mathbb{J}_{\nu }^{(n)};\varsigma )=\sum_{k=0}^{m}\left( 
\begin{array}{c}
m \\ 
k%
\end{array}%
\right) \mu _{k}x^{k}
\end{equation*}%
are all hyperbolic. However, observe that the Jensen polynomials of $\overset%
{\sim }{W}_{\nu ,n}(z)=W_{\nu ,n}(2\sqrt{z})$ are simply%
\begin{equation*}
P_{m}\left( \overset{\sim }{W}_{\nu ,n};\varsigma \right) =aP_{m}\left( 
\mathbb{J}_{\nu }^{(n)};\varsigma \right) -\varsigma P_{m}^{\prime }\left( 
\mathbb{J}_{\nu }^{(n)};\varsigma \right) .
\end{equation*}%
Lemma \ref{Le1} implies that all zeros of $P_{m}\left( \overset{\sim }{W}%
_{\nu ,n};\varsigma \right) $ are real and positive and that the smallest
one precedes the first zero of $P_{m}\left( \mathbb{J}_{\nu
}^{(n)};\varsigma \right) $. In view of Lemma \ref{Le2}, the latter
conclusion immediately yields that $\overset{\sim }{W}_{\nu ,n}\in \mathcal{%
LPI}$ and that its first zero precedes $j_{\nu ,1}^{(n)}$. Finally, the
first part of the statement of the lemma follows after we go back from $%
\overset{\sim }{W}_{\nu ,n}$ to $W_{\nu ,n}$ by setting $\varsigma =\frac{%
z^{2}}{4}.$
\end{proof}

\subsection{Euler-Rayleigh Sums for Positive Zeros of $J_{\protect\nu %
}^{(n)}(z)$}

Baricz et al. \cite{Ba4} proved Mittag-Leffler expansion of $J_{\nu
}^{(n)}(z)$ as follows 
\begin{equation}
J_{\nu }^{(n)}(z)=\frac{z^{\nu -n}}{2^{\nu }\Gamma (\nu +1-n)}%
\dprod\limits_{m\geq 1}\left( 1-\frac{z^{2}}{\left( j_{\nu ,m}^{(n)}\right)
^{2}}\right)  \label{r11}
\end{equation}%
where $j_{\nu ,m}^{(n)}$ is the $m$th positive zero of $J_{\nu }^{(n)}(z)$ $%
(m\in 
\mathbb{N}
,$ $n\in 
\mathbb{N}
_{0})$. Therefore we can write 
\begin{eqnarray}
g_{\nu ,n}(z) &=&2^{\nu }\Gamma (\nu -n+1)z^{1+n-\nu }J_{\nu }^{(n)}(z)
\label{S1} \\
&=&z\dprod\limits_{m\geq 1}\left( 1-\frac{z^{2}}{\left( j_{\nu
,m}^{(n)}\right) ^{2}}\right) .  \notag
\end{eqnarray}%
On the other hand, the series representation of $g_{\nu ,n}(z)$%
\begin{equation}
g_{\nu ,n}(z)=\sum_{m=0}^{\infty }\frac{\left( -1\right) ^{m}\Gamma (2m+\nu
+1)\Gamma (\nu -n+1)}{m!4^{m}\Gamma (2m-n+\nu +1)\Gamma (m+\nu +1)}z^{2m+1}.
\label{S2}
\end{equation}

Now, we would like to mention that by using the equations (\ref{S1}) and (%
\ref{S2}) we can obtain the following Euler-Rayleigh sums for the positive
zeros of the function $g_{\nu ,n}$. From the equality (\ref{S2}) we have%
\begin{equation}
g_{\nu ,n}(z)=z-\frac{\nu +2}{4(\nu -n+2)(\nu -n+1)}z^{3}+\frac{(\nu +4)(\nu
+3)}{32(\nu -n+4)(\nu -n+3)(\nu -n+2)(\nu -n+1)}z^{5}-\cdots .  \label{S3}
\end{equation}%
Now, if we consider (\ref{S1}), then some calculations yield that%
\begin{equation}
g_{\nu ,n}(z)=z-\dsum\limits_{m\geq 1}\frac{1}{\left( j_{\nu
,m}^{(n)}\right) ^{2}}z^{3}+\frac{1}{2}\left( \left( \dsum\limits_{m\geq 1}%
\frac{1}{\left( j_{\nu ,m}^{(n)}\right) ^{2}}\right)
^{2}-\dsum\limits_{m\geq 1}\frac{1}{\left( j_{\nu ,m}^{(n)}\right) ^{4}}%
\right) z^{5}-\cdots .  \label{S4}
\end{equation}%
By equating the first few coefficients with the same degrees in equations (%
\ref{S3}) and (\ref{S4}) we get,%
\begin{equation}
\dsum\limits_{m\geq 1}\frac{1}{\left( j_{\nu ,m}^{(n)}\right) ^{2}}=\frac{%
\nu +2}{4(\nu -n+2)(\nu -n+1)}  \label{S5}
\end{equation}%
and%
\begin{equation}
\dsum\limits_{m\geq 1}\frac{1}{\left( j_{\nu ,m}^{(n)}\right) ^{4}}=\frac{1}{%
16(\nu -n+2)(\nu -n+1)}\left( \frac{(\nu +2)^{2}}{(\nu -n+2)(\nu -n+1)}-%
\frac{(\nu +4)(\nu +3)}{(\nu -n+4)(\nu -n+3)}\right) .  \label{S6}
\end{equation}%
Here, it is important mentioning that for $n=0$ the equations (\ref{S5}) and
(\ref{S6}) reduce to%
\begin{equation*}
\dsum\limits_{m\geq 1}\frac{1}{\left( j_{\nu ,m}\right) ^{2}}=\frac{1}{4(\nu
+1)}\text{ \ \ and \ \ }\dsum\limits_{m\geq 1}\frac{1}{\left( j_{\nu
,m}\right) ^{4}}=\frac{1}{16(\nu +2)(\nu +1)^{2}}
\end{equation*}%
respectively, where $j_{\nu ,m}$ denotes the $m$th zero of classical Bessel
function $J_{\nu }$.

Another special case for $n=1,2$ the equations (\ref{S5}) and (\ref{S6})
reduce to%
\begin{equation*}
\dsum\limits_{m\geq 1}\frac{1}{\left( j_{\nu ,m}^{\prime }\right) ^{2}}=%
\frac{\nu +2}{4\nu (\nu +1)}\text{ \ \ and \ \ }\dsum\limits_{m\geq 1}\frac{1%
}{\left( j_{\nu ,m}^{\prime }\right) ^{4}}=\frac{\nu ^{2}+8\nu +8}{16\nu
^{2}(\nu +1)^{2}(\nu +2)}
\end{equation*}%
and%
\begin{equation*}
\dsum\limits_{m\geq 1}\frac{1}{\left( j_{\nu ,m}^{\prime \prime }\right) ^{2}%
}=\frac{\nu +2}{4(\nu -1)\nu }\text{ \ \ and \ \ }\dsum\limits_{m\geq 1}%
\frac{1}{\left( j_{\nu ,m}^{\prime \prime }\right) ^{4}}=\frac{13\nu
^{3}+19\nu ^{2}+26\nu +8}{16(\nu -1)^{2}\nu ^{2}(\nu +1)(\nu +2)}
\end{equation*}%
where $j_{\nu ,m}^{\prime }$ and $j_{\nu ,m}^{\prime \prime }$ denotes the $%
m $th zeros of function $J_{\nu }^{\prime }$ and $J_{\nu }^{\prime \prime },$
respectively.

\section{Main \ Results}

\subsection{Radii of Starlikeness and Convexity of The Functions $f_{\protect%
\nu ,n},$ $g_{\protect\nu ,n}$ and $h_{\protect\nu ,n}$}

The first principal result we established concerns the radii of starlikeness
and reads as follows. Here and in the sequel $I_{\nu }$ denotes the modified
Bessel function of the first kind and order $\nu .$ Note that $I_{\nu
}(z)=i^{-\nu }J_{\nu }(iz)$

\begin{theorem}
\label{t2} The following statements hold:

\begin{description}
\item[a)] If $\nu >n$ and $\beta \in \lbrack 0,1),$ then $r_{\beta }^{\ast
}(f_{\nu ,n})=x_{\nu ,1}^{(n)},$ where $x_{\nu ,1}^{(n)}$ is the smallest
positive root of the equation 
\begin{equation*}
\frac{rJ_{\nu }^{(n+1)}(r)}{(\nu -n)J_{\nu }^{(n)}(r)}-\beta =0.
\end{equation*}%
Moreover, if $n-1<\nu <n$ and $\beta \in \lbrack 0,1),$ then we have $%
r_{\beta }^{\ast }(f_{\nu ,n})=x_{\nu ,2}^{(n)},$ where $x_{\nu ,2}^{(n)}$
is the smallest positive root of the equation%
\begin{equation*}
\frac{rI_{\nu }^{(n+1)}(r)}{(\nu -n)I_{\nu }^{(n)}(r)}-\beta =0.
\end{equation*}

\item[b)] If $\nu >n-1$ and $\beta \in \lbrack 0,1),$ then $r_{\beta }^{\ast
}(g_{\nu ,n})=y_{\nu ,1}^{(n)},$ where $y_{\nu ,1}^{(n)}$ is the smallest
positive root of the equation 
\begin{equation*}
\frac{rJ_{\nu }^{(n+1)}(r)}{J_{\nu }^{(n)}(r)}+n+1-\nu -\beta =0.
\end{equation*}

\item[c)] If $\nu >n-1$ and $\beta \in \lbrack 0,1),$ then $r_{\beta }^{\ast
}(h_{\nu ,n})=z_{\nu ,1}^{(n)},$ where $z_{\nu ,1}^{(n)}$ is the smallest
positive root of the equation 
\begin{equation*}
\frac{\sqrt{r}J_{\nu }^{(n+1)}(\sqrt{r})}{J_{\nu }^{(n)}(\sqrt{r})}+n+2-\nu
-2\beta =0.
\end{equation*}
\end{description}
\end{theorem}

\begin{proof}
Firstly, we prove part \textbf{a }for\textbf{\ }$\nu >n$ and \textbf{b} and 
\textbf{c} for $\nu >n-1.$ We need to show that the following inequalities%
\begin{equation}
\func{Re}\left( \frac{zf_{\nu ,n}^{\prime }(z)}{f_{\nu ,n}(z)}\right) >\beta
,\text{ \ \ }\func{Re}\left( \frac{zg_{\nu ,n}^{\prime }(z)}{g_{\nu ,n}(z)}%
\right) >\beta \text{ \ \ and \ \ }\func{Re}\left( \frac{zh_{\nu ,n}^{\prime
}(z)}{h_{\nu ,n}(z)}\right) >\beta  \label{r1}
\end{equation}%
are valid for $z\in \mathbb{D}_{r_{\beta }^{\ast }(f_{\nu ,n})},$ $z\in 
\mathbb{D}_{r_{\beta }^{\ast }(g_{\nu ,n})}$ and $z\in \mathbb{D}_{r_{\beta
}^{\ast }(h_{\nu ,n})},$ respectively, and each of the above inequalities
does not hold in larger disks.

When we write the equation (\ref{r11}) in definition of the functions $%
f_{\nu ,n}(z),$ $g_{\nu ,n}(z)$ and $h_{\nu ,n}(z)$ we get by using
logarithmic derivation 
\begin{eqnarray*}
\frac{zf_{\nu ,n}^{\prime }(z)}{f_{\nu ,n}(z)} &=&\frac{1}{\nu -n}\frac{%
zJ_{\nu }^{(n+1)}(z)}{J_{\nu }^{(n)}(z)}=1-\frac{1}{\nu -n}%
\dsum\limits_{m\geq 1}\frac{2z^{2}}{\left( j_{\nu ,m}^{(n)}\right) ^{2}-z^{2}%
},\text{ \ \ }(\nu >n), \\
\frac{zg_{\nu ,n}^{\prime }(z)}{g_{\nu ,n}(z)} &=&n+1-\nu +\frac{zJ_{\nu
}^{(n+1)}(z)}{J_{\nu }^{(n)}(z)}=1-\dsum\limits_{m\geq 1}\frac{2z^{2}}{%
\left( j_{\nu ,m}^{(n)}\right) ^{2}-z^{2}},\text{ \ \ }(\nu >n-1), \\
\frac{zh_{\nu ,n}^{\prime }(z)}{h_{\nu ,n}(z)} &=&1+\frac{n-\nu }{2}+\frac{1%
}{2}\frac{\sqrt{z}J_{\nu }^{(n+1)}(\sqrt{z})}{J_{\nu }^{(n)}(\sqrt{z})}%
=1-\dsum\limits_{m\geq 1}\frac{z}{\left( j_{\nu ,m}^{(n)}\right) ^{2}-z},%
\text{ \ \ }(\nu >n-1).
\end{eqnarray*}%
It is known \cite{Ba1} that if $z\in 
\mathbb{C}
$ and $\lambda \in 
\mathbb{R}
$ are such that $\lambda >\left\vert z\right\vert ,$ then%
\begin{equation}
\frac{\left\vert z\right\vert }{\lambda -\left\vert z\right\vert }\geq \func{%
Re}\left( \frac{z}{\lambda -z}\right) .  \label{r2}
\end{equation}%
Then the inequality%
\begin{equation*}
\frac{\left\vert z\right\vert ^{2}}{\left( j_{\nu ,m}^{(n)}\right)
^{2}-\left\vert z\right\vert ^{2}}\geq \func{Re}\left( \frac{z^{2}}{\left(
j_{\nu ,m}^{(n)}\right) ^{2}-z^{2}}\right)
\end{equation*}%
holds for every $\nu >n-1.$ Therefore,%
\begin{eqnarray*}
\func{Re}\left( \frac{zf_{\nu ,n}^{\prime }(z)}{f_{\nu ,n}(z)}\right) &=&1-%
\frac{1}{\nu -n}\dsum\limits_{m\geq 1}\func{Re}\left( \frac{2z^{2}}{\left(
j_{\nu ,m}^{(n)}\right) ^{2}-z^{2}}\right) \geq 1-\frac{1}{\nu -n}%
\dsum\limits_{m\geq 1}\frac{2\left\vert z\right\vert ^{2}}{\left( j_{\nu
,m}^{(n)}\right) ^{2}-\left\vert z\right\vert ^{2}}=\frac{\left\vert
z\right\vert f_{\nu ,n}^{\prime }(\left\vert z\right\vert )}{f_{\nu
,n}(\left\vert z\right\vert )}, \\
\func{Re}\left( \frac{zg_{\nu ,n}^{\prime }(z)}{g_{\nu ,n}(z)}\right)
&=&1-\dsum\limits_{m\geq 1}\func{Re}\left( \frac{2z^{2}}{\left( j_{\nu
,m}^{(n)}\right) ^{2}-z^{2}}\right) \geq 1-\dsum\limits_{m\geq 1}\frac{%
2\left\vert z\right\vert ^{2}}{\left( j_{\nu ,m}^{(n)}\right)
^{2}-\left\vert z\right\vert ^{2}}=\frac{\left\vert z\right\vert g_{\nu
,n}^{\prime }(\left\vert z\right\vert )}{g_{\nu ,n}(\left\vert z\right\vert )%
}, \\
\func{Re}\left( \frac{zh_{\nu ,n}^{\prime }(z)}{h_{\nu ,n}(z)}\right)
&=&1-\dsum\limits_{m\geq 1}\func{Re}\left( \frac{z}{\left( j_{\nu
,m}^{(n)}\right) ^{2}-z}\right) \geq 1-\dsum\limits_{m\geq 1}\frac{%
\left\vert z\right\vert }{\left( j_{\nu ,m}^{(n)}\right) ^{2}-\left\vert
z\right\vert }=\frac{\left\vert z\right\vert h_{\nu ,n}^{\prime }(\left\vert
z\right\vert )}{h_{\nu ,n}(\left\vert z\right\vert )},
\end{eqnarray*}%
where equalities are attained only when $z=\left\vert z\right\vert =r.$ The
latest inequalities and the minimum principle for harmonic functions imply
that the corresponding inequalities in (\ref{r1}) hold if only if $%
\left\vert z\right\vert <x_{\nu ,1}^{(n)},$ $\left\vert z\right\vert <y_{\nu
,1}^{(n)}$ and $\left\vert z\right\vert <z_{\nu ,1}^{(n)},$ respectively,
where $x_{\nu ,1}^{(n)},$ $y_{\nu ,1}^{(n)}$ and $z_{\nu ,1}^{(n)}$ is the
smallest positive roots of the equations 
\begin{equation*}
\frac{rf_{\nu ,n}^{\prime }(r)}{f_{\nu ,n}(r)}=\beta ,\text{ \ \ }\frac{%
rg_{\nu ,n}^{\prime }(r)}{g_{\nu ,n}(r)}=\beta \text{ \ \ and \ \ }\frac{%
rh_{\nu ,n}^{\prime }(r)}{h_{\nu ,n}(r)}=\beta ,
\end{equation*}%
which are equivalent to 
\begin{equation*}
\frac{rJ_{\nu }^{(n+1)}(r)}{(\nu -n)J_{\nu }^{(n)}(r)}-\beta =0,\text{ \ \ }%
\frac{rJ_{\nu }^{(n+1)}(r)}{J_{\nu }^{(n)}(r)}+n+1-\nu -\beta =0
\end{equation*}%
and%
\begin{equation*}
\frac{\sqrt{r}J_{\nu }^{(n+1)}(\sqrt{r})}{J_{\nu }^{(n)}(\sqrt{r})}+n+2-\nu
-2\beta =0.
\end{equation*}%
The result follows from Lemma \ref{Le3} by taking instead of $a$ the values $%
\frac{(\beta -1)(\nu -n)}{2}$, $\frac{\beta -1}{2}$ and $\beta -1$,
respectively. In other words, Lemma \ref{Le3} show that all the zeros of the
above three functions are real and their first positive zeros do not exceed
the first positive zeros $j_{v,1}^{(n)}$ and $\sqrt{j_{v,1}^{(n)}}$. This
guarantees that the above inequalities hold. This completes the proof of
part \textbf{a} when $\nu >n$, and parts \textbf{b} and \textbf{c} when $\nu
>n-1$.

Now, to prove the statement for part \textbf{a }when $\nu \in (n-1,n),$ we
use the counterpart of (\ref{r2}), that is, 
\begin{equation}
\func{Re}\left( \frac{z}{\lambda -z}\right) \geq \frac{-\left\vert
z\right\vert }{\lambda +\left\vert z\right\vert },  \label{r3}
\end{equation}%
which holds for all $z\in 
\mathbb{C}
$ and $\lambda \in 
\mathbb{R}
$ are such that $\lambda >\left\vert z\right\vert $ (see \cite{Baricz}). If
in the inequality (\ref{r3}), we replace $z$ by $z^{2}$ and $\lambda $ by $%
\left( j_{\nu ,m}^{(n)}\right) ^{2},$ it follows that 
\begin{equation*}
\func{Re}\left( \frac{z^{2}}{\left( j_{\nu ,m}^{(n)}\right) ^{2}-z^{2}}%
\right) \geq \frac{-\left\vert z\right\vert ^{2}}{\left( j_{\nu
,m}^{(n)}\right) ^{2}+\left\vert z\right\vert ^{2}},
\end{equation*}%
provided that $\left\vert z\right\vert <j_{\nu ,1}^{(n)}$. Thus, for $%
n-1<\nu <n$ we obtain%
\begin{equation*}
\func{Re}\left( \frac{zf_{\nu ,n}^{\prime }(z)}{f_{\nu ,n}(z)}\right) =1-%
\frac{1}{\nu -n}\dsum\limits_{m\geq 1}\func{Re}\left( \frac{2z^{2}}{\left(
j_{\nu ,m}^{(n)}\right) ^{2}-z^{2}}\right) \geq 1+\frac{1}{\nu -n}%
\dsum\limits_{m\geq 1}\frac{2\left\vert z\right\vert ^{2}}{\left( j_{\nu
,m}^{(n)}\right) ^{2}+\left\vert z\right\vert ^{2}}=\frac{i\left\vert
z\right\vert f_{\nu ,n}^{\prime }(i\left\vert z\right\vert )}{f_{\nu
,n}(i\left\vert z\right\vert )}.
\end{equation*}%
In this case equality is attained if $z=i\left\vert z\right\vert =ir.$
Moreover, the latter inequality implies that%
\begin{equation*}
\func{Re}\left( \frac{zf_{\nu ,n}^{\prime }(z)}{f_{\nu ,n}(z)}\right) >\beta
\end{equation*}%
if and only if $\left\vert z\right\vert <x_{v,2}^{(n)}$, where $%
x_{v,2}^{(n)} $ denotes the smallest positive root of the equations%
\begin{equation*}
\frac{irf_{\nu ,n}^{\prime }(ir)}{f_{\nu ,n}(ir)}=\beta
\end{equation*}%
which is equivalent to 
\begin{equation*}
\frac{irJ_{\nu }^{(n+1)}(ir)}{(\nu -n)J_{\nu }^{(n)}(ir)}-\beta =0\text{ or
\ }\frac{rI_{\nu }^{(n+1)}(r)}{(\nu -n)I_{\nu }^{(n)}(r)}-\beta =0
\end{equation*}%
for $n-1<\nu <n.$ It follows from Lemma \ref{Le3} that the first positive
zero of $z\mapsto irJ_{\nu }^{(n+1)}(ir)-\beta \left( \nu -n\right) J_{\nu
}^{(n)}(ir)$ does not exceed $j_{\nu ,1}^{(n)}$ which guarantees that the
above inequalities are valid. All we need to prove is that the above
function has actually only one zero in $(0,\infty )$. Observe that,
according to Lemma \ref{l11}, the function%
\begin{equation*}
r\mapsto \frac{irJ_{\nu }^{(n+1)}(ir)}{J_{\nu }^{(n)}(ir)}=\frac{%
\sum_{m=0}^{\infty }\frac{(2m-n+\nu )\Gamma (2m+\nu +1)}{m!2^{2m+\nu }\Gamma
(2m-n+\nu +1)\Gamma (m+\nu +1)}r^{2m}}{\sum_{m=0}^{\infty }\frac{\Gamma
(2m+\nu +1)}{m!2^{2m+\nu }\Gamma (2m-n+\nu +1)\Gamma (m+\nu +1)}r^{2m}}
\end{equation*}%
is increasing on $(0,\infty )$ as a quotient of two power series whose
positive coefficients form the increasing \textquotedblleft quotient
sequence\textquotedblright\ $\{2m-n+\nu \}_{m\geq 0}.$ On the other hand,
the above function tends to $\nu -n$ when $r\rightarrow 0$, so that its
graph can intersect the horizontal line $y=\beta \left( \nu -n\right) >\nu
-n $ only once. This completes the proof of part \textbf{a} of the theorem
when $\nu \in (n-1,n)$.
\end{proof}

With regards to Theorem \ref{t2}, we tabulate the radius of starlikeness for 
$f_{\nu ,n}$, $g_{\nu ,n}$ and $h_{\nu ,n}$ for a fixed $\nu =2.5,$ $%
n=0,1,2,3$ and respectively $\beta =0$ and $\beta =0.5$. These are given in
Table 1. Also in Table 1, we see that radius of starlikeness is decreasing
according to the order of derivative and the order of starlikeness. On the
other words, from all these results we concluded that$\ r_{\beta }^{\ast
}(f_{\nu ,0})>r_{\beta }^{\ast }(f_{\nu ,1})>r_{\beta }^{\ast }(f_{\nu
,2})>\cdots >r_{\beta }^{\ast }(f_{\nu ,n})>\cdots $ for $\beta \in \lbrack
0,1)$ and $\nu >n-1,$ $n\in 
\mathbb{N}
_{0}.$ In addition to, we can write $r_{\beta _{1}}^{\ast }(f_{\nu
,n})<r_{\beta _{0}}^{\ast }(f_{\nu ,n})$ for $0\leq \beta _{0}<\beta _{1}<1$
and $\nu >n-1,$ $n\in 
\mathbb{N}
_{0}.$ Same inequalities is also true for$\ r_{\beta }^{\ast }(g_{\nu ,n})$
and\ $r_{\beta }^{\ast }(h_{\nu ,n}).$

For $n=0$ in the Theorem \ref{t2} we obtain the results of Baricz et al \cite%
{Baricz}. Our results is a common generalization of these results.

\begin{eqnarray*}
&&%
\begin{tabular}{l|l|l|l|l|l|l|}
\cline{2-7}
& \multicolumn{2}{l}{$\ \ \ \ \ \ \ r_{\beta }^{\ast }(f_{2.5,n})$} & 
\multicolumn{2}{|l}{$\ \ \ \ \ \ \ r_{\beta }^{\ast }(g_{2.5,n})$} & 
\multicolumn{2}{|l|}{$\ \ \ \ \ \ \ r_{\beta }^{\ast }(h_{2.5,n})$} \\ 
\cline{2-7}
& $\beta =0$ & $\beta =0.5$ & $\beta =0$ & $\beta =0.5$ & $\beta =0$ & $%
\beta =0.5$ \\ \hline
\multicolumn{1}{|l|}{$n=0$} & $\ 3.6328$ & $\ 2.7569$ & $2.5011$ & $1.8192$
& $11.1696$ & $6.2556$ \\ \hline
\multicolumn{1}{|l|}{$n=1$} & $2.1056$ & $1.5926$ & $1.7975$ & $1.3307$ & $%
5.4265$ & $3.2312$ \\ \hline
\multicolumn{1}{|l|}{$n=2$} & $0.8512$ & $0.6229$ & $1.1285$ & $0.8512$ & $%
2.0284$ & $1.2735$ \\ \hline
\multicolumn{1}{|l|}{$n=3$} & $0.4586$ & $0.3051$ & $0.4819$ & $0.3703$ & $%
0.3543$ & $0.2323$ \\ \hline
\end{tabular}
\\
&& \\
&&\text{\textbf{Table1.} Radii of starlikeness for }f_{\nu ,n},g_{\nu ,n}%
\text{ and }h_{\nu ,n}\text{ when }\nu =2.5
\end{eqnarray*}

The second principal result we established concerns the radii of convexity
and reads as follows.

\begin{theorem}
\label{t3} The following statements hold:

\begin{description}
\item[a)] If $\nu >n$ and $\beta \in \lbrack 0,1),$ then\ the radius $%
r_{\beta }^{c}(f_{\nu ,n})$ is the smallest positive root of the equation 
\begin{equation*}
1-\beta +\frac{rJ_{\nu }^{(n+2)}(r)}{J_{\nu }^{(n+1)}(r)}+\left( \frac{1}{%
\nu -n}-1\right) \frac{rJ_{\nu }^{(n+1)}(r)}{J_{\nu }^{(n)}(r)}=0.
\end{equation*}%
Moreover, $r_{\beta }^{c}(f_{\nu ,n})<j_{\nu ,1}^{(n+1)}<j_{\nu ,1}^{(n)}.$

\item[b)] If $\nu >n-1$ and $\beta \in \lbrack 0,1),$ then\ the radius $%
r_{\beta }^{c}(g_{\nu ,n})$ is the smallest positive root of the equation 
\begin{equation*}
n+1-\nu -\beta +\frac{(n-\nu +2)rJ_{\nu }^{(n+1)}(r)+r^{2}J_{\nu }^{(n+2)}(r)%
}{(n-\nu +1)J_{\nu }^{(n)}(r)+rJ_{\nu }^{(n+1)}(r)}=0.
\end{equation*}

\item[c)] If $\nu >n-1$ and $\beta \in \lbrack 0,1),$ then\ the radius $%
r_{\beta }^{c}(h_{\nu ,n})$ is the smallest positive root of the equation 
\begin{equation*}
\frac{n+2-\nu -2\beta }{2}+\frac{\sqrt{r}}{2}\frac{(n-\nu +3)J_{\nu
}^{(n+1)}(\sqrt{r})+\sqrt{r}J_{\nu }^{(n+2)}(\sqrt{r})}{(n-\nu +2)J_{\nu
}^{(n)}(\sqrt{r})+\sqrt{r}J_{\nu }^{(n+1)}(\sqrt{r})}=0.
\end{equation*}
\end{description}
\end{theorem}

\begin{proof}
\textbf{a)} Since%
\begin{equation*}
1+\frac{zf_{\nu ,n}^{\prime \prime }(z)}{f_{\nu ,n}^{\prime }(z)}=1+\frac{%
zJ_{\nu }^{(n+2)}(z)}{J_{\nu }^{(n+1)}(z)}+\left( \frac{1}{\nu -n}-1\right) 
\frac{zJ_{\nu }^{(n+1)}(z)}{J_{\nu }^{(n)}(z)}
\end{equation*}%
and by means of (\ref{r11}) we have%
\begin{equation*}
\frac{zJ_{\nu }^{(n+1)}(z)}{J_{\nu }^{(n)}(z)}=\nu -n-\dsum\limits_{m\geq 1}%
\frac{2z^{2}}{\left( j_{\nu ,m}^{(n)}\right) ^{2}-z^{2}}
\end{equation*}%
it follows that%
\begin{equation*}
1+\frac{zf_{\nu ,n}^{\prime \prime }(z)}{f_{\nu ,n}^{\prime }(z)}=1-\left( 
\frac{1}{\nu -n}-1\right) \dsum\limits_{m\geq 1}\frac{2z^{2}}{\left( j_{\nu
,m}^{(n)}\right) ^{2}-z^{2}}-\dsum\limits_{m\geq 1}\frac{2z^{2}}{\left(
j_{\nu ,m}^{(n+1)}\right) ^{2}-z^{2}}.
\end{equation*}%
Now, suppose that $\nu \in (n,n+1].$ By using the inequality (\ref{r2}), for
all $z\in \mathbb{D}_{j_{\nu ,1}^{(n)}}$ we obtain the inequality%
\begin{eqnarray*}
\func{Re}\left( 1+\frac{zf_{\nu ,n}^{\prime \prime }(z)}{f_{\nu ,n}^{\prime
}(z)}\right) &=&1-\left( \frac{1}{\nu -n}-1\right) \dsum\limits_{m\geq 1}%
\func{Re}\left( \frac{2z^{2}}{\left( j_{\nu ,m}^{(n)}\right) ^{2}-z^{2}}%
\right) -\dsum\limits_{m\geq 1}\func{Re}\left( \frac{2z^{2}}{\left( j_{\nu
,m}^{(n+1)}\right) ^{2}-z^{2}}\right) \\
&\geq &1-\left( \frac{1}{\nu -n}-1\right) \dsum\limits_{m\geq 1}\frac{2r^{2}%
}{\left( j_{\nu ,m}^{(n)}\right) ^{2}-r^{2}}-\dsum\limits_{m\geq 1}\frac{%
2r^{2}}{\left( j_{\nu ,m}^{(n+1)}\right) ^{2}-r^{2}}
\end{eqnarray*}%
where $\left\vert z\right\vert =r.$ Moreover, observe that if we use the
inequality \cite[Lemma 2.1]{Ba1}%
\begin{equation*}
\mu \func{Re}\left( \frac{z}{a-z}\right) -\func{Re}\left( \frac{z}{b-z}%
\right) \geq \mu \frac{\left\vert z\right\vert }{a-\left\vert z\right\vert }-%
\frac{\left\vert z\right\vert }{b-\left\vert z\right\vert }
\end{equation*}%
where $a>b>0,$ $\mu \in \lbrack 0,1]$ and $z\in 
\mathbb{C}
$ such that $|z|<b$, then we get that the above inequality is also valid
when $\nu >n+1$. Here we used that the zeros of the $n$th and $(n+1)$th
derivative of $J_{v}$ are interlacing according to Lemma\ref{l1}. The above
inequality implies for $r\in (0,j_{\nu ,1}^{(n)})$%
\begin{equation*}
\underset{z\in \mathbb{D}_{r}}{\inf }\left\{ \func{Re}\left( 1+\frac{zf_{\nu
,n}^{\prime \prime }(z)}{f_{\nu ,n}^{\prime }(z)}\right) \right\} =1+\frac{%
rf_{\nu ,n}^{\prime \prime }(r)}{f_{\nu ,n}^{\prime }(r)}.
\end{equation*}%
On the other hand, we define the function $\varphi _{\nu ,n}:(n,j_{\nu
,1}^{(n)})\rightarrow 
\mathbb{R}
,$ 
\begin{equation*}
\varphi _{\nu ,n}(r)=1+\frac{rf_{\nu ,n}^{\prime \prime }(r)}{f_{\nu
,n}^{\prime }(r)}.
\end{equation*}%
Since the zeros of the $n$th and $(n+1)$th derivative of $J_{v}$ are
interlacing according to Lemma \ref{l1} and $r<j_{\nu ,1}^{(n+1)}<j_{\nu
,1}^{(n)}$ $\left( \text{or }r<\sqrt{j_{\nu ,1}^{(n)}j_{\nu ,1}^{(n+1)}}%
\right) $ for all $\nu >n$ we have 
\begin{equation*}
\left( j_{\nu ,m}^{(n)}\right) \left( \left( j_{\nu ,m}^{(n+1)}\right)
^{2}-r^{2}\right) -\left( j_{\nu ,m}^{(n+1)}\right) \left( \left( j_{\nu
,m}^{(n)}\right) ^{2}-r^{2}\right) <0.
\end{equation*}%
Thus following inequality 
\begin{eqnarray*}
\frac{d\varphi _{\nu ,n}(r)}{dr} &=&-\left( \frac{1}{\nu -n}-1\right)
\dsum\limits_{m\geq 1}\frac{4r\left( j_{\nu ,m}^{(n)}\right) ^{2}}{\left(
\left( j_{\nu ,m}^{(n)}\right) ^{2}-r^{2}\right) ^{2}}-\dsum\limits_{m\geq 1}%
\frac{4r\left( j_{\nu ,m}^{(n+1)}\right) ^{2}}{\left( \left( j_{\nu
,m}^{(n+1)}\right) ^{2}-r^{2}\right) ^{2}} \\
&<&\dsum\limits_{m\geq 1}\frac{4r\left( j_{\nu ,m}^{(n)}\right) ^{2}}{\left(
\left( j_{\nu ,m}^{(n)}\right) ^{2}-r^{2}\right) ^{2}}-\dsum\limits_{m\geq 1}%
\frac{4r\left( j_{\nu ,m}^{(n+1)}\right) ^{2}}{\left( \left( j_{\nu
,m}^{(n+1)}\right) ^{2}-r^{2}\right) ^{2}} \\
&=&4r\dsum\limits_{m\geq 1}\frac{\left( j_{\nu ,m}^{(n)}\right) ^{2}\left(
\left( j_{\nu ,m}^{(n+1)}\right) ^{2}-r^{2}\right) ^{2}-\left( j_{\nu
,m}^{(n+1)}\right) ^{2}\left( \left( j_{\nu ,m}^{(n)}\right)
^{2}-r^{2}\right) ^{2}}{\left( \left( j_{\nu ,m}^{(n)}\right)
^{2}-r^{2}\right) ^{2}\left( \left( j_{\nu ,m}^{(n+1)}\right)
^{2}-r^{2}\right) ^{2}}<0
\end{eqnarray*}%
is satisfied. Consequently, the function $\varphi _{\nu ,n}$ is strictly
decreasing. Observe also that $\lim_{r\searrow 0}\varphi _{\nu
,n}(r)=1>\beta $ and $\lim_{r\nearrow j_{\nu ,1}^{(n)}}\varphi _{\nu
,n}(r)=-\infty ,$ which means that for $z\in \mathbb{D}_{r_{1}}$ we have%
\begin{equation*}
\func{Re}\left( 1+\frac{zf_{\nu ,n}^{\prime \prime }(z)}{f_{\nu ,n}^{\prime
}(z)}\right) >\beta
\end{equation*}%
if and ony if $r_{1}$ is the unique root of 
\begin{equation*}
1+\frac{rf_{\nu ,n}^{\prime \prime }(r)}{f_{\nu ,n}^{\prime }(r)}=\beta ,
\end{equation*}%
situated in $(0,j_{\nu ,1}^{(n)}).$

\textbf{b)} Observe that%
\begin{equation*}
1+\frac{zg_{\nu ,n}^{\prime \prime }(z)}{g_{\nu ,n}^{\prime }(z)}=(n-\nu +1)+%
\frac{(n-\nu +2)zJ_{\nu }^{(n+1)}(z)+z^{2}J_{\nu }^{(n+2)}(z)}{(n-\nu
+1)J_{\nu }^{(n)}(z)+zJ_{\nu }^{(n+1)}(z)}.
\end{equation*}

By using (\ref{F1}) and (\ref{r11}) we have that%
\begin{eqnarray}
g_{\nu ,n}^{\prime }(z) &=&2^{\nu }\Gamma (\nu -n+1)z^{n-\nu }\left[ (n-\nu
+1)J_{\nu }^{(n)}(z)+zJ_{\nu }^{(n+1)}(z)\right]  \label{r31} \\
&=&\sum_{m=0}^{\infty }\frac{\left( -1\right) ^{m}(2m+1)\Gamma (2m+\nu
+1)\Gamma (\nu -n+1)}{m!\Gamma (2m-n+\nu +1)\Gamma (m+\nu +1)}\left( \frac{z%
}{2}\right) ^{2m}  \notag
\end{eqnarray}%
and%
\begin{equation*}
\underset{m\rightarrow \infty }{\lim }\frac{m\log m}{%
\begin{array}{c}
\lbrack 2m\log 2+\log \Gamma (m+1)+\log \Gamma (2m-n+\nu +1)+\log \Gamma
(m+\nu +1) \\ 
-\log \Gamma (2m+\nu +1)-\log \Gamma (\nu -n+1)-\log (2m+1)]%
\end{array}%
}=\frac{1}{2}.
\end{equation*}%
Here, we used $m!=\Gamma (m+1)$ and $\underset{m\rightarrow \infty }{\lim }%
\frac{\log \Gamma (am+b)}{m\log m}=a,$ where $b$ and $c$ are positive
constants. So, by applying Hadamard's Theorem \cite[p. 26]{Le} we can write
the infinite product representation of $g_{\nu ,n}^{\prime }(z)$ as follows:%
\begin{equation}
g_{\nu ,n}^{\prime }(z)=\dprod\limits_{m\geq 1}\left( 1-\frac{z^{2}}{\left(
\gamma _{\nu ,m}^{(n)}\right) ^{2}}\right) ,  \label{r4}
\end{equation}%
where $\gamma _{\nu ,m}^{(n)}$ denotes the $m$th positive zero of the
function $g_{\nu ,n}^{\prime }.$ From Lemma \ref{Le3} for $\nu >n-1$ the
function $g_{\nu ,n}^{\prime }\in \mathcal{LP}$, and the smallest positive
zero of $g_{\nu ,n}^{\prime }$ does not exceed the first positive zero of $%
J_{\nu }^{(n)}.$

By means of (\ref{r4}) we have 
\begin{equation*}
1+\frac{zg_{\nu ,n}^{\prime \prime }(z)}{g_{\nu ,n}^{\prime }(z)}%
=1-\dsum\limits_{m\geq 1}\frac{2z^{2}}{\left( \gamma _{\nu ,m}^{(n)}\right)
^{2}-z^{2}}.
\end{equation*}%
By using the inequality (\ref{r2}), for all $z\in \mathbb{D}_{\gamma _{\nu
,m}^{(n)}}$ we obtain the inequality%
\begin{equation*}
\func{Re}\left( 1+\frac{zg_{\nu ,n}^{\prime \prime }(z)}{g_{\nu ,n}^{\prime
}(z)}\right) \geq 1-\dsum\limits_{m\geq 1}\frac{2r^{2}}{\left( \gamma _{\nu
,m}^{(n)}\right) ^{2}-r^{2}}
\end{equation*}%
where $\left\vert z\right\vert =r.$ Thus, for $r\in (0,\gamma _{\nu
,1}^{(n)})$ we get%
\begin{equation*}
\underset{z\in \mathbb{D}_{r}}{\inf }\left\{ \func{Re}\left( 1+\frac{zg_{\nu
,n}^{\prime \prime }(z)}{g_{\nu ,n}^{\prime }(z)}\right) \right\} =1+\frac{%
rg_{\nu ,n}^{\prime \prime }(r)}{g_{\nu ,n}^{\prime }(r)}.
\end{equation*}%
The function $G_{\nu ,n}:(0,\gamma _{\nu ,1}^{(n)})\rightarrow 
\mathbb{R}
,$ defined by%
\begin{equation*}
G_{\nu ,n}(r)=1+\frac{rg_{\nu ,n}^{\prime \prime }(r)}{g_{\nu ,n}^{\prime
}(r)},
\end{equation*}%
is strictly decreasing and $\lim_{r\searrow 0}G_{\nu ,n}(r)=1>\beta $ and $%
\lim_{r\nearrow \gamma _{\nu ,1}^{(n)}}G_{\nu ,n}(r)=-\infty $.
Consequently, in view of the minimum principle for harmonic functions for $%
z\in \mathbb{D}_{r_{2}}$ we have that%
\begin{equation*}
\func{Re}\left( 1+\frac{zg_{\nu ,n}^{\prime \prime }(z)}{g_{\nu ,n}^{\prime
}(z)}\right) >\beta
\end{equation*}%
if and ony if $r_{2}$ is the unique root of 
\begin{equation*}
1+\frac{rg_{\nu ,n}^{\prime \prime }(r)}{g_{\nu ,n}^{\prime }(r)}=\beta ,
\end{equation*}%
situated in $(0,\gamma _{\nu ,1}^{(n)}).$

\textbf{c)} Observe that%
\begin{equation*}
1+\frac{zh_{\nu ,n}^{\prime \prime }(z)}{h_{\nu ,n}^{\prime }(z)}=\frac{%
n-\nu +2}{2}+\frac{\sqrt{z}}{2}\frac{(n-\nu +3)J_{\nu }^{(n+1)}(\sqrt{z})+%
\sqrt{z}J_{\nu }^{(n+2)}(\sqrt{z})}{(n-\nu +2)J_{\nu }^{(n)}(\sqrt{z})+\sqrt{%
z}J_{\nu }^{(n+1)}(\sqrt{z})}.
\end{equation*}

By using (\ref{F1}) and (\ref{r11}) we have that%
\begin{eqnarray}
h_{\nu ,n}^{\prime }(z) &=&2^{\nu -1}\Gamma (\nu -n+1)z^{\frac{n-\nu }{2}}%
\left[ (n-\nu +2)J_{\nu }^{(n)}(\sqrt{z})+\sqrt{z}J_{\nu }^{(n+1)}(\sqrt{z})%
\right]  \label{r32} \\
&=&\sum_{m=0}^{\infty }\frac{\left( -1\right) ^{m}(m+1)\Gamma (2m+\nu
+1)\Gamma (\nu -n+1)}{m!\Gamma (2m-n+\nu +1)\Gamma (m+\nu +1)}\left( \frac{z%
}{4}\right) ^{m}  \notag
\end{eqnarray}%
and%
\begin{equation*}
\underset{m\rightarrow \infty }{\lim }\frac{m\log m}{%
\begin{array}{c}
\lbrack 2m\log 2+\log \Gamma (m+1)+\log \Gamma (2m-n+\nu +1)+\log \Gamma
(m+\nu +1) \\ 
-\log \Gamma (2m+\nu +1)-\log \Gamma (\nu -n+1)-\log (m+1)]%
\end{array}%
}=\frac{1}{2}.
\end{equation*}%
So, by applying Hadamard's Theorem \cite[p. 26]{Le} we can write the
infinite product representation of $h_{\nu ,n}^{\prime }(z)$ as follows:%
\begin{equation}
h_{\nu ,n}^{\prime }(z)=\dprod\limits_{m\geq 1}\left( 1-\frac{z}{\delta
_{\nu ,m}^{(n)}}\right) ,  \label{r41}
\end{equation}%
where $\delta _{\nu ,m}^{(n)}$ denotes the $m$th positive zero of the
function $h_{\nu ,n}^{\prime }.$ From Lemma \ref{Le3} for $\nu >n-1$ the
function $h_{\nu ,n}^{\prime }\in \mathcal{LP}$, and the smallest positive
zero of $h_{\nu ,n}^{\prime }$ does not exceed the first positive zero of $%
J_{\nu }^{(n)}.$

By means of (\ref{r4}) we have 
\begin{equation*}
1+\frac{zh_{\nu ,n}^{\prime \prime }(z)}{h_{\nu ,n}^{\prime }(z)}%
=1-\dsum\limits_{m\geq 1}\frac{z}{\delta _{\nu ,m}^{(n)}-z}.
\end{equation*}%
By using the inequality (\ref{r2}), for all $z\in \mathbb{D}_{\delta _{\nu
,m}^{(n)}}$ we obtain the inequality%
\begin{equation*}
\func{Re}\left( 1+\frac{zh_{\nu ,n}^{\prime \prime }(z)}{h_{\nu ,n}^{\prime
}(z)}\right) \geq 1-\dsum\limits_{m\geq 1}\frac{r}{\delta _{\nu ,m}^{(n)}-r},
\end{equation*}%
where $\left\vert z\right\vert =r.$ Thus, for $r\in (0,\delta _{\nu
,1}^{(n)})$ we get%
\begin{equation*}
\underset{z\in \mathbb{D}_{r}}{\inf }\left\{ \func{Re}\left( 1+\frac{zh_{\nu
,n}^{\prime \prime }(z)}{h_{\nu ,n}^{\prime }(z)}\right) \right\} =1+\frac{%
rh_{\nu ,n}^{\prime \prime }(r)}{h_{\nu ,n}^{\prime }(r)}.
\end{equation*}%
The function $H_{\nu ,n}:(0,\delta _{\nu ,1}^{(n)})\rightarrow 
\mathbb{R}
,$ defined by%
\begin{equation*}
H_{\nu ,n}(r)=1+\frac{rh_{\nu ,n}^{\prime \prime }(r)}{h_{\nu ,n}^{\prime
}(r)},
\end{equation*}%
is strictly decreasing and $\lim_{r\searrow 0}H_{\nu ,n}(r)=1>\beta $ and $%
\lim_{r\nearrow \delta _{\nu ,1}^{(n)}}H_{\nu ,n}(r)=-\infty $.
Consequently, in view of the minimum principle for harmonic functions for $%
z\in \mathbb{D}_{r_{3}}$ we have that%
\begin{equation*}
\func{Re}\left( 1+\frac{zh_{\nu ,n}^{\prime \prime }(z)}{h_{\nu ,n}^{\prime
}(z)}\right) >\beta
\end{equation*}%
if and ony if $r_{3}$ is the unique root of 
\begin{equation*}
1+\frac{rh_{\nu ,n}^{\prime \prime }(r)}{h_{\nu ,n}^{\prime }(r)}=\beta ,
\end{equation*}%
situated in $(0,\delta _{\nu ,1}^{(n)}).$
\end{proof}

With regards to Theorem \ref{t3}, we tabulate the radius of convexity for $%
f_{\nu ,n}$, $g_{\nu ,n}$ and $h_{\nu ,n}$ for a fixed $\nu =3.5,$ $%
n=0,1,2,3 $ and respectively $\beta =0$ and $\beta =0.5$. These are given in
Table 2. Also in Table 2, we see that radius of convexity is decreasing
according to the order of derivative and the order of convexity. On the
other words, from all these results we concluded that$\ r_{\beta
}^{c}(f_{\nu ,0})>r_{\beta }^{c}(f_{\nu ,1})>r_{\beta }^{c}(f_{\nu
,2})>\cdots >r_{\beta }^{c}(f_{\nu ,n})>\cdots $ for $\beta \in \lbrack 0,1)$
and $\nu >n-1,$ $n\in 
\mathbb{N}
_{0}.$ In addition to, we can write $r_{\beta _{1}}^{c}(f_{\nu ,n})<r_{\beta
_{0}}^{c}(f_{\nu ,n})$ for $0\leq \beta _{0}<\beta _{1}<1$ and $\nu >n-1,$ $%
n\in 
\mathbb{N}
_{0}.$ Same inequalities is also true for$\ r_{\beta }^{c}(g_{\nu ,n})$ and\ 
$r_{\beta }^{c}(h_{\nu ,n}).$

\begin{eqnarray*}
&&%
\begin{tabular}{l|l|l|l|l|l|l|}
\cline{2-7}
& \multicolumn{2}{l}{$\ \ \ \ \ \ \ r_{\beta }^{c}(f_{3.5,n})$} & 
\multicolumn{2}{|l}{$\ \ \ \ \ \ \ r_{\beta }^{c}(g_{3.5,n})$} & 
\multicolumn{2}{|l|}{$\ \ \ \ \ \ \ r_{\beta }^{c}(h_{3.5,n})$} \\ 
\cline{2-7}
& $\beta =0$ & $\beta =0.5$ & $\beta =0$ & $\beta =0.5$ & $\beta =0$ & $%
\beta =0.5$ \\ \hline
\multicolumn{1}{|l|}{$n=0$} & $2.7183$ & $2.0865$ & $0.5234$ & $1.1461$ & $%
6.2189$ & $3.7194$ \\ \hline
\multicolumn{1}{|l|}{$n=1$} & $1.8179$ & $1.3998$ & $1.2017$ & $0.9084$ & $%
3.7394$ & $2.2873$ \\ \hline
\multicolumn{1}{|l|}{$n=2$} & $1.0592$ & $0.8123$ & $0.8833$ & $0.6715$ & $%
1.9450$ & $1.2190$ \\ \hline
\multicolumn{1}{|l|}{$n=3$} & $0.4141$ & $0.3131$ & $0.5683$ & $0.4350$ & $%
0.7726$ & $0.4968$ \\ \hline
\end{tabular}
\\
&& \\
&&\text{\textbf{Table2.} Radii of convexity for }f_{\nu ,n},g_{\nu ,n}\text{
and }h_{\nu ,n}\text{ when }\nu =3.5
\end{eqnarray*}

For $n=0$ in the Theorem \ref{t3} we obtain the results of Baricz and Sz\'{a}%
sz \cite{Ba1}. Our results is a common generalization of these results.

\subsection{Bounds for Radii of Starlikeness and Convexity of The Functions $%
g_{\protect\nu ,n}$ and $h_{\protect\nu ,n}$}

In this subsection we consider two different functions $g_{\nu ,n}$ and $%
h_{\nu ,n}$ which are normalized forms of the Bessel function derivatives of
the first kind given by (\ref{F1}) . Here firstly our aim is to show that
the radii of univalence of these functions correspond to the radii of
starlikeness.

\begin{theorem}
\label{t4} The following statements hold:

\begin{description}
\item[a)] If $\nu >n-1,$ then $r^{\ast }(g_{\nu ,n})$ satisfies the
inequalities 
\begin{equation*}
r^{\ast }(g_{\nu ,n})<\sqrt{\frac{2(\nu -n+2)(\nu -n+1)}{\nu +2},}
\end{equation*}%
\begin{equation*}
2\sqrt{\frac{(\nu -n+2)(\nu -n+1)}{3(\nu +2)}}<r^{\ast }(g_{\nu ,n})<2\sqrt{%
\frac{1}{\frac{3(\nu +2)}{(\nu -n+2)(\nu -n+1)}-\frac{5(\nu +4)(\nu +3)}{%
3(\nu -n+4)(\nu -n+3)(\nu +2)}}}.
\end{equation*}

\item[b)] If $\nu >n-1,$ then $r^{\ast }(h_{\nu ,n})$ satisfies the
inequalities 
\begin{equation*}
r^{\ast }(h_{\nu ,n})<\frac{2(\nu -n+2)(\nu -n+1)}{\nu +2},
\end{equation*}%
\begin{equation*}
\frac{2(\nu -n+2)(\nu -n+1)}{\nu +2}<r^{\ast }(h_{\nu ,n})<\frac{2}{\frac{%
\nu +2}{(\nu -n+2)(\nu -n+1)}-\frac{3(\nu +4)(\nu +3)}{4(\nu -n+4)(\nu
-n+3)(\nu +2)}}.
\end{equation*}
\end{description}
\end{theorem}

\begin{proof}
\textbf{a)} By using the first Rayleigh sum (\ref{S5}) and the implict
relation for $r^{\ast }(g_{\nu ,n}),$ obtained by Kreyszing and Todd \cite%
{Kr}, we get for all $\nu >n-1$ that%
\begin{equation*}
\frac{1}{\left( r^{\ast }(g_{\nu ,n})\right) ^{2}}=\dsum\limits_{m\geq 1}%
\frac{2}{\left( j_{\nu ,m}^{(n)}\right) ^{2}-\left( r^{\ast }(g_{\nu
,n})\right) ^{2}}>\dsum\limits_{m\geq 1}\frac{2}{\left( j_{\nu
,m}^{(n)}\right) ^{2}}=\frac{\nu +2}{2(\nu -n+2)(\nu -n+1)}.
\end{equation*}%
Now, by using the Euler-Rayleigh inequalities it is possible to have more
tight bounds for the radius of univalence (and starlikeness) $r^{\ast
}(g_{\nu ,n})$. We define the function $\Psi _{\nu ,n}(z)=g_{\nu ,n}^{\prime
}(z)$, where $g_{\nu ,n}^{\prime }$ defined by (\ref{r4}). Now, taking
logarithmic derivative of both sides of (\ref{r4}) for $\left\vert
z\right\vert <\gamma _{\nu ,1}^{(n)}$ we have%
\begin{equation}
\frac{\Psi _{\nu ,n}^{\prime }(z)}{\Psi _{\nu ,n}(z)}=-\dsum\limits_{m\geq 1}%
\frac{2z}{\left( \gamma _{\nu ,m}^{(n)}\right) ^{2}-z^{2}}%
=-2\dsum\limits_{m\geq 1}\dsum\limits_{k\geq 0}\frac{1}{\left( \gamma _{\nu
,m}^{(n)}\right) ^{2(k+1)}}z^{2k+1}=-2\dsum\limits_{k\geq 0}\sigma
_{k+1}z^{2k+1}  \label{p1}
\end{equation}%
where $\sigma _{k}=\sum_{m\geq 1}\left( \gamma _{\nu ,m}^{(n)}\right) ^{-k}$
is Euler-Rayleigh sum for the zeros of $\Psi _{\nu ,n}.$ Also, using (\ref%
{r31}) from the infinite sum representation of $\Psi _{\nu ,n}$ we obtain%
\begin{equation}
\frac{\Psi _{\nu ,n}^{\prime }(z)}{\Psi _{\nu ,n}(z)}=\frac{%
\dsum\limits_{m\geq 0}U_{m}z^{2m+1}}{\dsum\limits_{m\geq 0}V_{m}z^{2m}},
\label{p2}
\end{equation}%
where%
\begin{equation*}
U_{m}=\frac{2\left( -1\right) ^{m+1}\Gamma (2m+\nu +3)\Gamma (\nu -n+1)(2m+3)%
}{m!4^{m+1}\Gamma (2m-n+\nu +3)\Gamma (m+\nu +2)}
\end{equation*}%
and%
\begin{equation*}
V_{m}=\frac{\left( -1\right) ^{m}\Gamma (2m+\nu +1)\Gamma (\nu -n+1)(2m+1)}{%
m!4^{m}\Gamma (2m-n+\nu +1)\Gamma (m+\nu +1)}.
\end{equation*}%
By comparing the coefficients with the same degrees of (\ref{p1}) and (\ref%
{p2}) we obtain the Euler-Rayleigh sums%
\begin{equation*}
\sigma _{1}=\frac{3(\nu +2)}{4(\nu -n+2)(\nu -n+1)}
\end{equation*}%
and%
\begin{equation*}
\sigma _{2}=\frac{3(\nu +2)}{16(\nu -n+2)(\nu -n+1)}\left( \frac{3(\nu +2)}{%
(\nu -n+2)(\nu -n+1)}-\frac{5(\nu +4)(\nu +3)}{3(\nu -n+4)(\nu -n+3)(\nu +2)}%
\right)
\end{equation*}%
By using the Euler-Rayleigh inequalities%
\begin{equation*}
\sigma _{k}^{-\frac{1}{k}}<\left( \gamma _{\nu ,m}^{(n)}\right) ^{2}<\frac{%
\sigma _{k}}{\sigma _{k+1}}
\end{equation*}%
for $\nu >n-1$, $k\in 
\mathbb{N}
$ and $k=1$ we get the following inequality%
\begin{equation*}
\frac{4(\nu -n+2)(\nu -n+1)}{3(\nu +2)}<\left( r^{\ast }(g_{\nu ,n})\right)
^{2}<\frac{4}{\frac{3(\nu +2)}{(\nu -n+2)(\nu -n+1)}-\frac{5(\nu +4)(\nu +3)%
}{3(\nu -n+4)(\nu -n+3)(\nu +2)}}
\end{equation*}%
and it is possible to have more tighter bounds for other values of $k\in 
\mathbb{N}
.$

\textbf{b)} By using the first Rayleigh sum (\ref{S5}) and the implict
relation for $r^{\ast }(h_{\nu ,n}),$ obtained by Kreyszing and Todd \cite%
{Kr}, we get for all $\nu >n-1$ that%
\begin{equation*}
\frac{1}{r^{\ast }(h_{\nu ,n})}=\dsum\limits_{m\geq 1}\frac{1}{\left( j_{\nu
,m}^{(n)}\right) ^{2}-r^{\ast }(h_{\nu ,n})}>\dsum\limits_{m\geq 1}\frac{1}{%
\left( j_{\nu ,m}^{(n)}\right) ^{2}}=\frac{\nu +2}{2(\nu -n+2)(\nu -n+1)}.
\end{equation*}%
Now, by using the Euler-Rayleigh inequalities it is possible to have more
tight bounds for the radius of univalence (and starlikeness) $r^{\ast
}(h_{\nu ,n})$. We define the function $\Phi _{\nu ,n}(z)=h_{\nu ,n}^{\prime
}(z)$, where $h_{\nu ,n}^{\prime }$ defined by (\ref{r32}) or (\ref{r41}).
Now, taking logarithmic derivative of both sides of (\ref{r41}) we have%
\begin{equation}
\frac{\Phi _{\nu ,n}^{\prime }(z)}{\Phi _{\nu ,n}(z)}=-\dsum\limits_{m\geq 1}%
\frac{1}{\delta _{\nu ,m}^{(n)}-z}=-\dsum\limits_{m\geq
1}\dsum\limits_{k\geq 0}\frac{1}{\left( \delta _{\nu ,m}^{(n)}\right) ^{k+1}}%
z^{k}=-\dsum\limits_{k\geq 0}\rho _{k+1}z^{k},\text{ \ \ }\left\vert
z\right\vert <\delta _{\nu ,1}^{(n)}  \label{p11}
\end{equation}%
where $\rho _{k}=\sum_{m\geq 1}\left( \delta _{\nu ,m}^{(n)}\right) ^{-k}$
is Euler-Rayleigh sum for the zeros of $\Phi _{\nu ,n}.$ Also, using (\ref%
{r32}) from the infinite sum representation of $\Phi _{\nu ,n}$ we obtain%
\begin{equation}
\frac{\Phi _{\nu ,n}^{\prime }(z)}{\Phi _{\nu ,n}(z)}=\frac{%
\dsum\limits_{m\geq 0}P_{m}z^{m}}{\dsum\limits_{m\geq 0}Q_{m}z^{m}},
\label{p12}
\end{equation}%
where%
\begin{equation*}
P_{m}=\frac{\left( -1\right) ^{m+1}\Gamma (2m+\nu +3)\Gamma (\nu -n+1)(m+2)}{%
m!4^{m+1}\Gamma (2m-n+\nu +3)\Gamma (m+\nu +2)}
\end{equation*}%
and%
\begin{equation*}
Q_{m}=\frac{\left( -1\right) ^{m}\Gamma (2m+\nu +1)\Gamma (\nu -n+1)(m+1)}{%
m!4^{m}\Gamma (2m-n+\nu +1)\Gamma (m+\nu +1)}.
\end{equation*}%
By comparing the coefficients with the same degrees of (\ref{p11}) and (\ref%
{p12}) we obtain the Euler-Rayleigh sums%
\begin{equation*}
\rho _{1}=\frac{\nu +2}{2(\nu -n+2)(\nu -n+1)}
\end{equation*}%
and%
\begin{equation*}
\rho _{2}=\frac{\nu +2}{4(\nu -n+2)(\nu -n+1)}\left( \frac{\nu +2}{(\nu
-n+2)(\nu -n+1)}-\frac{3(\nu +4)(\nu +3)}{4(\nu -n+4)(\nu -n+3)(\nu +2)}%
\right)
\end{equation*}%
By using the Euler-Rayleigh inequalities%
\begin{equation*}
\rho _{k}^{-\frac{1}{k}}<\delta _{\nu ,m}^{(n)}<\frac{\rho _{k}}{\rho _{k+1}}
\end{equation*}%
for $\nu >n-1$, $k\in 
\mathbb{N}
$ and $k=1$ we get the following inequality%
\begin{equation*}
\frac{2(\nu -n+2)(\nu -n+1)}{\nu +2}<r^{\ast }(h_{\nu ,n})<\frac{2}{\frac{%
\nu +2}{(\nu -n+2)(\nu -n+1)}-\frac{3(\nu +4)(\nu +3)}{4(\nu -n+4)(\nu
-n+3)(\nu +2)}}
\end{equation*}%
and it is possible to have more tighter bounds for other values of $k\in 
\mathbb{N}
.$
\end{proof}

If we take $n=0$ in the Theorem \ref{t4} we obtain the results of Akta\c{s}
et al. \cite{Ak}. Our results is a common generalization of these results.
For special cases of parameters $\nu $ and $n,$ Theorem \ref{t4} reduces
tight lower and upper bounds for the radii of starlikeness and convexity of
many elemanter functions. For example for $\nu =\frac{3}{2}$ and $n=2$ in
Theorem \ref{t4} we have $\sqrt{\frac{2}{7}}<r^{\ast }\left( g_{\frac{3}{2}%
,2}(z)=4\sin z-4z\cos z\right) <\sqrt{\frac{3}{7}}$ and $\frac{3}{7}<r^{\ast
}\left( h_{\frac{3}{2},2}(z)=4\sqrt{z}\sin \sqrt{z}-4z\cos \sqrt{z}\right) <%
\frac{2940}{5969}.$

The next result concerning bounds for radii of convexity of functions $%
g_{\nu ,n}$ and $h_{\nu ,n}$.

\begin{theorem}
\label{t5} The following statements hold:

\begin{description}
\item[a)] If $\nu >n-1,$ then $r^{c}(g_{\nu ,n})$ satisfies the inequalities 
\begin{equation*}
\frac{2}{3}\sqrt{\frac{(\nu -n+2)(\nu -n+1)}{(\nu +2)}}<r^{c}(g_{\nu ,n})<2%
\sqrt{\frac{1}{\frac{9(\nu +2)}{(\nu -n+2)(\nu -n+1)}-\frac{25(\nu +4)(\nu
+3)}{9(\nu -n+4)(\nu -n+3)(\nu +2)}}}.
\end{equation*}

\item[b)] If $\nu >n-1,$ then $r^{c}(h_{\nu ,n})$ satisfies the inequalities 
\begin{equation*}
\frac{(\nu -n+2)(\nu -n+1)}{\nu +2}<r^{c}(h_{\nu ,n})<\frac{1}{\frac{\nu +2}{%
(\nu -n+2)(\nu -n+1)}-\frac{9(\nu +4)(\nu +3)}{16(\nu -n+4)(\nu -n+3)(\nu +2)%
}}.
\end{equation*}
\end{description}
\end{theorem}

\begin{proof}
\textbf{a)} By using the Alexander duality theorem for starlike and convex
functions we can say that the function $g_{\nu ,n}(z)$ is convex if and only
if $zg_{\nu ,n}^{\prime }(z)$ is starlike. But, the smallest positive zero
of $z\mapsto z\left( zg_{\nu ,n}^{\prime }(z)\right) ^{\prime }$ is actually
the radius of starlikeness of $z\mapsto \left( zg_{\nu ,n}^{\prime
}(z)\right) $, according to Theorem \ref{t2} and Theorem \ref{t3}.
Therefore, the radius of convexity $r^{c}(g_{\nu ,n})$ is the smallest
positive root of the equation $\left( zg_{\nu ,n}^{\prime }(z)\right)
^{\prime }=0$. Therefore from (\ref{r31}), we have%
\begin{equation*}
\Delta _{\nu ,n}(z)=\left( zg_{\nu ,n}^{\prime }(z)\right) ^{\prime
}=\sum_{m=0}^{\infty }\frac{\left( -1\right) ^{m}(2m+1)^{2}\Gamma (2m+\nu
+1)\Gamma (\nu -n+1)}{m!4^{m}\Gamma (2m-n+\nu +1)\Gamma (m+\nu +1)}z^{2m}.
\end{equation*}

Since the function $g_{\nu ,n}(z)$ belongs to the Laguerre-P\'{o}lya class
of entire functions and $\mathcal{LP}$ is closed under differentiation, we
can say that the function $\Delta _{\nu ,n}(z)\in \mathcal{LP}$. Therefore,
the zeros of the function $\Delta _{\nu ,n}$ are all real. Suppose that $%
d_{\nu ,m}^{(n)}$ are the zeros of the function $\Delta _{\nu ,n}$. Then the
function $\Delta _{\nu ,n}$ has the infinite product representation as
follows:%
\begin{equation}
\Delta _{\nu ,n}(z)=\dprod\limits_{m\geq 1}\left( 1-\frac{z^{2}}{\left(
d_{\nu ,m}^{(n)}\right) ^{2}}\right) .  \label{p13}
\end{equation}%
By taking the logarithmic derivative of (\ref{p13}) we get%
\begin{equation}
\frac{\Delta _{\nu ,n}^{\prime }(z)}{\Delta _{\nu ,n}(z)}=-2\dsum\limits_{m%
\geq 1}\frac{z}{\left( d_{\nu ,m}^{(n)}\right) ^{2}-z^{2}}%
=-2\dsum\limits_{m\geq 1}\dsum\limits_{k\geq 0}\frac{1}{\left( d_{\nu
,m}^{(n)}\right) ^{2(k+1)}}z^{2k+1}=-2\dsum\limits_{k\geq 0}\kappa
_{k+1}z^{2k+1},\text{ \ \ }\left\vert z\right\vert <d_{\nu ,1}^{(n)}
\label{p14}
\end{equation}%
where $\kappa _{k}=\sum_{m\geq 1}\left( d_{\nu ,m}^{(n)}\right) ^{-k}$ is
Euler-Rayleigh sum for the zeros of $\Delta _{\nu ,n}.$ On the other hand,
by considering infinite sum representation of $\Delta _{\nu ,n}(z)$ we obtain%
\begin{equation}
\frac{\Delta _{\nu ,n}^{\prime }(z)}{\Delta _{\nu ,n}(z)}=\frac{%
\dsum\limits_{m\geq 0}X_{m}z^{2m+1}}{\dsum\limits_{m\geq 0}Y_{m}z^{2m}},
\label{p15}
\end{equation}%
where%
\begin{equation*}
X_{m}=\frac{2\left( -1\right) ^{m+1}\Gamma (2m+\nu +3)\Gamma (\nu
-n+1)(2m+3)^{2}}{m!4^{m+1}\Gamma (2m-n+\nu +3)\Gamma (m+\nu +2)}
\end{equation*}%
and%
\begin{equation*}
Y_{m}=\frac{\left( -1\right) ^{m}\Gamma (2m+\nu +1)\Gamma (\nu
-n+1)(2m+1)^{2}}{m!4^{m}\Gamma (2m-n+\nu +1)\Gamma (m+\nu +1)}.
\end{equation*}%
By comparing the coefficients of (\ref{p14}) and (\ref{p15}) we obtain%
\begin{equation*}
\kappa _{1}=\frac{9(\nu +2)}{4(\nu -n+2)(\nu -n+1)}
\end{equation*}%
and%
\begin{equation*}
\kappa _{2}=\frac{9(\nu +2)}{16(\nu -n+2)(\nu -n+1)}\left( \frac{9(\nu +2)}{%
(\nu -n+2)(\nu -n+1)}-\frac{25(\nu +4)(\nu +3)}{9(\nu -n+4)(\nu -n+3)(\nu +2)%
}\right)
\end{equation*}%
By using the Euler-Rayleigh inequalities%
\begin{equation*}
\kappa _{k}^{-\frac{1}{k}}<\left( d_{\nu ,m}^{(n)}\right) ^{2}<\frac{\kappa
_{k}}{\kappa _{k+1}}
\end{equation*}%
for $\nu >n-1$, $k\in 
\mathbb{N}
$ and $k=1$ we get the following inequality%
\begin{equation*}
\frac{4(\nu -n+2)(\nu -n+1)}{9(\nu +2)}<\left( r^{c}(g_{\nu ,n})\right) ^{2}<%
\frac{4}{\frac{9(\nu +2)}{(\nu -n+2)(\nu -n+1)}-\frac{25(\nu +4)(\nu +3)}{%
9(\nu -n+4)(\nu -n+3)(\nu +2)}}
\end{equation*}%
and it is possible to have more tighter bounds for other values of $k\in 
\mathbb{N}
.$

\textbf{b)} By using the same procedure as in the previous proof we can say
that the radius of convexity $r^{c}(h_{\nu ,n})$ is the smallest positive
root of the equation $\left( zh_{\nu ,n}^{\prime }(z)\right) ^{\prime }=0$
according to Theorem \ref{t3}. From (\ref{r32}), we have%
\begin{equation}
\Theta _{\nu ,n}(z)=\left( zh_{\nu ,n}^{\prime }(z)\right) ^{\prime
}=\sum_{m=0}^{\infty }\frac{\left( -1\right) ^{m}(m+1)^{2}\Gamma (2m+\nu
+1)\Gamma (\nu -n+1)}{m!4^{m}\Gamma (2m-n+\nu +1)\Gamma (m+\nu +1)}z^{m}.
\label{p151}
\end{equation}%
Moreover, we know $h_{\nu ,n}(z)$ belongs to the Laguerre-P\'{o}lya class of
entire functions and $\mathcal{LP}$, consequently $\Theta _{\nu ,n}(z)\in 
\mathcal{LP}$. On the other words, the zeros of the function $\Theta _{\nu
,n}$ are all real. Assume that $l_{\nu ,m}^{(n)}$ are the zeros of the
function $\Theta _{\nu ,n}$. In this case, the function $\Theta _{\nu ,n}$
has the infinite product representation as follows:%
\begin{equation}
\Theta _{\nu ,n}(z)=\dprod\limits_{m\geq 1}\left( 1-\frac{z^{2}}{\left(
l_{\nu ,m}^{(n)}\right) ^{2}}\right) .  \label{p16}
\end{equation}%
By taking the logarithmic derivative of both sides of (\ref{p16}) for \ $%
\left\vert z\right\vert <l_{\nu ,1}^{(n)}$ we have%
\begin{equation}
\frac{\Theta _{\nu ,n}^{\prime }(z)}{\Theta _{\nu ,n}(z)}=-\dsum\limits_{m%
\geq 1}\frac{1}{l_{\nu ,m}^{(n)}-z}=-\dsum\limits_{m\geq
1}\dsum\limits_{k\geq 0}\frac{1}{\left( l_{\nu ,m}^{(n)}\right) ^{k+1}}%
z^{k}=-\dsum\limits_{k\geq 0}\omega _{k+1}z^{k}  \label{p17}
\end{equation}%
where $\omega _{k}=\sum_{m\geq 1}\left( l_{\nu ,m}^{(n)}\right) ^{-k}$. In
addition, by using the derivative of infinite sum representation considering
infinite sum representation of (\ref{p151}) we obtain%
\begin{equation}
\frac{\Theta _{\nu ,n}^{\prime }(z)}{\Theta _{\nu ,n}(z)}=\dsum\limits_{m%
\geq 0}T_{m}z^{m}\diagup \dsum\limits_{m\geq 0}S_{m}z^{m},  \label{p18}
\end{equation}%
where%
\begin{equation*}
T_{m}=\frac{\left( -1\right) ^{m+1}\Gamma (2m+\nu +3)\Gamma (\nu
-n+1)(m+2)^{2}}{m!4^{m+1}\Gamma (2m-n+\nu +3)\Gamma (m+\nu +2)}
\end{equation*}%
and%
\begin{equation*}
S_{m}=\frac{\left( -1\right) ^{m}\Gamma (2m+\nu +1)\Gamma (\nu -n+1)(m+1)^{2}%
}{m!4^{m}\Gamma (2m-n+\nu +1)\Gamma (m+\nu +1)}.
\end{equation*}%
By comparing the coefficients of (\ref{p17}) and (\ref{p18}) we get%
\begin{equation*}
\omega _{1}=\frac{\nu +2}{(\nu -n+2)(\nu -n+1)}
\end{equation*}%
and%
\begin{equation*}
\omega _{2}=\frac{\nu +2}{(\nu -n+2)(\nu -n+1)}\left( \frac{\nu +2}{(\nu
-n+2)(\nu -n+1)}-\frac{9(\nu +4)(\nu +3)}{16(\nu -n+4)(\nu -n+3)(\nu +2)}%
\right)
\end{equation*}%
By using the Euler-Rayleigh inequalities%
\begin{equation*}
\omega _{k}^{-\frac{1}{k}}<l_{\nu ,m}^{(n)}<\frac{\omega _{k}}{\omega _{k+1}}
\end{equation*}%
for $\nu >n-1$, $k\in 
\mathbb{N}
$ and $k=1$ we get the following inequality%
\begin{equation*}
\frac{(\nu -n+2)(\nu -n+1)}{\nu +2}<r^{c}(h_{\nu ,n})<\frac{1}{\frac{\nu +2}{%
(\nu -n+2)(\nu -n+1)}-\frac{9(\nu +4)(\nu +3)}{16(\nu -n+4)(\nu -n+3)(\nu +2)%
}}
\end{equation*}

and it is possible to have more tighter bounds for other values of $k\in 
\mathbb{N}
.$
\end{proof}

If we take $n=0$ in the Theorem \ref{t5} we obtain the results of Akta\c{s}
et al. \cite{Ak1}. For special case $n=1,2,3$ we obtain following result.

\begin{corollary}
\label{c1} The following inequalities hold: 
\begin{eqnarray*}
\frac{2}{3}\sqrt{\frac{\nu (\nu +1)}{(\nu +2)}} &<&r^{c}(g_{\nu ,1})<2\sqrt{%
\frac{1}{\frac{9(\nu +2)}{\nu (\nu +1)}-\frac{25(\nu +4)(\nu +3)}{9(\nu
+3)(\nu +2)^{2}}}}\text{\ \ \ }\nu >0, \\
\text{ }\frac{\nu (\nu +1)}{(\nu +2)} &<&r^{c}(h_{\nu ,1})<\frac{1}{\frac{%
\nu +2}{\nu (\nu +1)}-\frac{9(\nu +4)(\nu +3)}{16(\nu +3)(\nu +2)^{2}}}\text{
\ \ }\nu >0, \\
\frac{2}{3}\sqrt{\frac{\nu (\nu -1)}{(\nu +2)}} &<&r^{c}(g_{\nu ,2})<2\sqrt{%
\frac{1}{\frac{9(\nu +2)}{\nu (\nu -1)}-\frac{25(\nu +4)(\nu +3)}{9(\nu
+1)(\nu +2)^{2}}}}\text{ \ \ }\nu >1, \\
\text{ }\frac{\nu (\nu -1)}{(\nu +2)} &<&r^{c}(h_{\nu ,2})<\frac{1}{\frac{%
\nu +2}{\nu (\nu -1)}-\frac{9(\nu +4)(\nu +3)}{16(\nu +1)(\nu +2)^{2}}}\text{
\ \ }\nu >1, \\
\frac{2}{3}\sqrt{\frac{(\nu -2)(\nu -1)}{(\nu +2)}} &<&r^{c}(g_{\nu ,3})<2%
\sqrt{\frac{1}{\frac{9(\nu +2)}{(\nu -2)(\nu -1)}-\frac{25(\nu +4)(\nu +3)}{%
9\nu (\nu +1)(\nu +2)}}}\text{ \ \ }\nu >2, \\
\frac{(\nu -2)(\nu -1)}{(\nu +2)} &<&r^{c}(h_{\nu ,3})<\frac{1}{\frac{\nu +2%
}{(\nu -2)(\nu -1)}-\frac{9(\nu +4)(\nu +3)}{16\nu (\nu +1)(\nu +2)}}\text{
\ \ }\nu >2.
\end{eqnarray*}
\end{corollary}

\end{document}